\documentclass[11pt]{amsart}
\usepackage[margin=1in]{geometry}

\usepackage{amsmath,amssymb,amsthm}
\usepackage[colorlinks=true,linkcolor=blue,citecolor=blue,urlcolor=blue]{hyperref}

\numberwithin{equation}{section}

\theoremstyle{plain}
\newtheorem{thm}{Theorem}[section]
\newtheorem{cor}[thm]{Corollary}
\newtheorem{lem}[thm]{Lemma}
\newtheorem{prop}[thm]{Proposition}

\theoremstyle{definition}
\newtheorem{defn}[thm]{Definition}

\theoremstyle{remark}
\newtheorem{rem}[thm]{Remark}

\newcommand{\Z}{\mathbb{Z}}
\newcommand{\Q}{\mathbb{Q}}
\newcommand{\R}{\mathbb{R}}

\newcommand{\F}{\mathbb{F}}
\newcommand{\SL}{\mathrm{SL}}
\newcommand{\GL}{\mathrm{GL}}
\newcommand{\PSL}{\mathrm{PSL}}
\newcommand{\PGL}{\mathrm{PGL}}
\newcommand{\tr}{\operatorname{Tr}}

\newcommand{\cO}{\mathcal{O}}
\newcommand{\fp}{\mathfrak{p}}

\newcommand{\GammaTwo}{\Gamma_2}

\begin{document}

\title[Failure of the congruence Hasse principle]{Failure of the congruence Hasse principle for commutators in
orientable surface groups}

\author{Jianchun Wu}
\address{Department of Mathematics, Soochow University, Suzhou 215006, CHINA}
\email{wujianchun@suda.edu.cn}



\thanks{The author acknowledges partial support from the National Natural Science Foundation of China (Grant No. 12271385).}

\begin{abstract}
Ghosh, Meiri and Sarnak showed that the commutator equation over $\SL_2(\Z)$
satisfies a profinite local to global principle, and that for a prime
$\ell\equiv\pm1\pmod5$ it fails for $\SL_2(\Z[1/\ell])$ with a
reciprocity obstruction. We extend the failure to
surface groups. For the genus-$2$ surface group
$\GammaTwo$ we construct an explicit representation
$$
\rho\colon\GammaTwo\longrightarrow \SL_2\bigl(\cO_{K,S}\bigr),
\qquad K=\Q(\sqrt5),\quad
\cO_{K,S}=\Z\Bigl[\tfrac{1+\sqrt5}{2},\tfrac1{10}\Bigr]
$$
whose quotients modulo the maximal ideals are the full groups
$\SL_2(\F_q)$, except in residue characteristic 3, where the quotient has order 24. For every
$k\ge2$ the word $[a_1,b_1]^k$ is not a commutator in $\GammaTwo$,
while its image is a commutator in every one of these quotients. The same holds for every surface group $\Gamma_g$
with $g\ge2$.
\end{abstract}


\maketitle

\section{Introduction}

The Hasse principle in number theory asks whether a solution over $\Q$
is guaranteed by solutions over all completions of $\Q$. Ghosh, Meiri
and Sarnak~\cite{GhoshMeiriSarnak} studied the commutator equation from
this point of view and took every finite quotient of the group as a
local condition. They call the resulting principle profinite. For
$\SL_2(\Z)$ it holds, and the proof uses all finite quotients rather
than the congruence ones. For $\SL_2(\Z[1/\ell])$ with a prime
$\ell\equiv\pm1\pmod5$ it fails.  For this ring essentially all finite quotients are congruence quotients. The obstruction is a reciprocity condition on
cubic Markoff surfaces. This raises the question whether similar
phenomena occur for other groups, such as surface groups.

Every closed orientable surface of genus $g\ge2$ carries a hyperbolic
metric. So its fundamental group, the surface group
$\Gamma_g=\langle a_1,b_1,\dots,a_g,b_g
\bigm|\prod_{i=1}^g[a_i,b_i]=1\rangle$, embeds in $\PSL_2(\R)$ as a
discrete subgroup. $\R$ is a field, so there are no proper ideals to
reduce modulo, and this representation gives no congruences. To run
the Hasse principle on a group one needs finite quotients. A
representation into $\SL_2$ over a ring with many maximal ideals,
such as $\Z$ or $\Z[1/\ell]$, supplies them.

Let $\rho\colon\Gamma\to\SL_2(\cO_{K,S})$ be such a homomorphism, where
$\cO_{K,S}$ is the ring of integers of a number field $K$ with the
primes of a finite set $S$ inverted. The
finite quotients of a surface group are not tied to $\SL_2$, so the
local conditions come from $\rho$. For $\SL_2(\Z)$ the local tests are the
congruences modulo the integers $m$, and $m$ factorises into prime
powers. In a number field the finite quotients of the ring of integers
are indexed by ideals, and ideals factor into prime ideals. The role
of a prime $p$ is played by a maximal ideal $\fp$ of $\cO_{K,S}$, and
reduction modulo $\fp$ gives the finite image $\rho_\fp(\Gamma)$. The
\emph{congruence Hasse principle
for commutators} with respect to $\rho$ asserts that a word
$w\in[\Gamma,\Gamma]$ is a commutator in $\Gamma$ as soon as its image
is a commutator in every $\rho_\fp(\Gamma)$. The principle depends on
$\rho$. A trivial $\rho$ would test nothing. One therefore wants a
representation whose congruence quotients are large.

For surface groups we construct an explicit representation
\[
\rho\colon \GammaTwo \longrightarrow \SL_2(\cO_{K,S}),
\qquad K=\Q(\sqrt5),\quad
\cO_{K,S}=\Z\Bigl[\frac{1+\sqrt5}{2},\frac1{10}\Bigr]
\]
which has a faithful generic
form: $\rho_t\colon\GammaTwo\to\SL_2(\cO_{K,S}[t,t^{-1}])$, where $t$
is an indeterminate (Appendix~\ref{app:faithful}). Every $\Gamma_g$ with
$g\ge2$ embeds in $\GammaTwo$ as a subgroup of index $g-1$, so the
construction restricts to every surface group. The ring of coefficients is small. It is obtained from $\Z$ by adjoining a single
quadratic irrationality and inverting only $2$ and $5$. Its quotients
modulo the maximal ideals are as large as possible: in residue
characteristics different
from $3$ they are the full groups $\SL_2(\F_q)$, and they separate the
points of $\rho(\GammaTwo)$ (Corollary~\ref{cor:separate}). This representation
yields a failure of the principle. For every $k\ge2$ the word
$[a_1,b_1]^k$ is not a commutator in $\GammaTwo$, but its image is a
commutator in $\rho_\fp(\GammaTwo)$ for every maximal ideal $\fp$.
Passing to a subgroup of index
$g-1$ gives the same conclusion for $\Gamma_g$, for every $g\ge2$.

The paper is organised as follows. Section~\ref{sec:mech} isolates the
mechanism: a group-theoretic and an arithmetic input that combine into a
general criterion, Theorem~\ref{thm:criterion}. Section~\ref{sec:repr}
constructs the representation and shows that its congruence quotients are
large. Section~\ref{sec:main} proves the failure
and its descent to every genus.

\section{The mechanism}\label{sec:mech}

The mechanism has two inputs. The group-theoretic input produces elements of
the commutator subgroup that are not commutators. The arithmetic input makes
the congruence quotients large enough that every element in them is a
commutator. Theorem~\ref{thm:criterion} combines the two.

\subsection{The group-theoretic input}

Let $F_2=\langle a,b\rangle$ be a free group of rank 2. The following
lemma is well known; we include a short proof for completeness.

\begin{lem}\label{lem:freenoncomm}
For $n\in\Z$ the element $[a,b]^n$ is a commutator in $F_2$ if and only if
$n\in\{0,1,-1\}$.
\end{lem}
\begin{proof}
$[a,b]^{1}$, $[a,b]^{-1}=[b,a]$, and $[a,b]^{0}=1$ are all commutators.
Let $\mathrm{cl}(w)$ denote the commutator length of $w$, the least $m$ such
that $w$ is a product of $m$ commutators. So $w$ is a commutator exactly
when $\mathrm{cl}(w)=1$. By Culler~\cite[Example~2.6]{Culler1981}, for $n\geq 1$, $\mathrm{cl}([a,b]^n)=\mathrm{cl}([b,a]^n)=\lfloor
n/2\rfloor+1$ where $\lfloor
\cdot\rfloor$ denotes the greatest integer function. Hence both $[a,b]^n$ and $[a,b]^{-n}=[b,a]^n$ are not commutators when $n\geq 2$.
\end{proof}

\begin{cor}\label{cor:retr}
Let $\GammaTwo=\langle a_1,b_1,a_2,b_2\mid [a_1,b_1][a_2,b_2]\rangle$ be the
fundamental group of an orientable closed surface of genus $2$. Then for every $k\ge2$ the element
$[a_1,b_1]^k$ is not a commutator in $\GammaTwo$.
\end{cor}
\begin{proof}
$\GammaTwo$ admits a retraction $r$ on to $F=\langle a_1,b_1\rangle$ defined by
$r(a_1)=a_1$, $r(b_1)=b_1$, $r(a_2)=b_1$ and $r(b_2)=a_1$; it is well defined,
since it carries the defining relation to $[a_1,b_1][b_1,a_1]=1$.  Suppose
$[a_1,b_1]^k=[u,v]$ in $\GammaTwo$. Then
$[a_1,b_1]^k=r([a_1,b_1]^k)=[r(u),r(v)]$ in $F\cong F_2$.
This is a contradiction by Lemma~\ref{lem:freenoncomm}.
\end{proof}

\subsection{The arithmetic input}

Let $K$ be a number field and let $S$ be a finite set of places of $K$.
The ring of $S$-integers $\cO_{K,S}$ consists of the elements of $K$ that
are integral at every finite place outside $S$. The residue field at any
maximal ideal of $\cO_{K,S}$ is a finite field $\F_q$. Its characteristic is the
unique rational prime below that maximal ideal.

Let $\SL_2(\cO_{K,S})$ be the group of
$2\times2$ matrices of determinant $1$ with entries in $\cO_{K,S}$. 
 For a maximal ideal $\fp$ of $\cO_{K,S}$, reduction
modulo $\fp$ preserves determinants, so it gives a homomorphism
\[
\pi_\fp\colon \SL_2(\cO_{K,S})\longrightarrow \SL_2(\F_q).
\]
For a homomorphism $\rho$ with values in $\SL_2(\cO_{K,S})$ we write
$\rho_\fp:=\pi_\fp\circ\rho$.

\begin{lem}\label{lem:thompson}
For every prime power $q\ge4$ every element of $\SL_2(\F_q)$ is a commutator.
\end{lem}
\begin{proof}
Since $|\F_q|\ge 4$, Thompson's theorem~\cite[Theorem 2]{Thompson1961}
shows that every non-scalar matrix in $\SL_2(\F_q)$ is a commutator.

It remains to treat the scalar matrices. These are $I$ and, when $q$ is
odd, $-I$. The identity is a commutator, namely $I=[I,I]$. So assume $q$
is odd and consider $-I$. Let
$$
S=\{x^2:x\in\F_q\},\qquad S'=\{-1-y^2:y\in\F_q\}.
$$
For $x,y\in\F_q^{\times}$ we have $x^2=y^2$ exactly when $y=\pm x$, and
$x\ne-x$ because $q$ is odd, so squaring is two-to-one on $\F_q^{\times}$.
Hence $\F_q$ has $(q-1)/2$ nonzero squares; with $0=0^2$ this gives
$|S|=(q+1)/2$. Since $S'=-1-S$ and $|-S|=|S|$,
$|S|+|S'|=q+1>q$. Thus $S$ and $S'$ are subsets of the $q$-element set
$\F_q$ with $|S|+|S'|>q$, and they meet. So $x^2=-1-y^2$ for some
$x,y\in\F_q$.

Then
$$
X=\begin{pmatrix}0&1\\-1&0\end{pmatrix},\qquad
Y=\begin{pmatrix}x&y\\y&-x\end{pmatrix}
$$
lie in $\SL_2(\F_q)$, since $\det X=\det Y=1$.
Hence $XY=-YX$, and
$$
[X,Y]=XYX^{-1}Y^{-1}=(-YX)X^{-1}Y^{-1}=-YY^{-1}=-I.
$$
So $-I$ is a commutator, and the lemma follows.
\end{proof}

\subsection{The principle and a criterion for its failure}

\begin{defn}\label{def:chp}
Let $\Gamma$ be a group and $\rho\colon\Gamma\to\SL_2(\cO_{K,S})$ a
homomorphism. For a maximal ideal $\fp$ the \emph{congruence quotient} of
$\Gamma$ at $\fp$ with respect to $\rho$ is the finite group
$\rho_\fp(\Gamma)\le\SL_2(\F_q)$. We say that $\Gamma$ has the
\emph{congruence Hasse principle for commutators} with respect to $\rho$ if the
following holds. For every $x\in[\Gamma,\Gamma]$, if the image of $x$ in
$\rho_\fp(\Gamma)$ is a commutator for every maximal ideal $\fp$, then $x$ is a
commutator in $\Gamma$.
\end{defn}

The two inputs above combine into the following criterion.

\begin{thm}\label{thm:criterion}
Let $\Gamma$ be a group and $\rho\colon\Gamma\to\SL_2(\cO_{K,S})$ a
homomorphism, and let $E$ be a set of maximal ideals of $\cO_{K,S}$. Let
$w\in[\Gamma,\Gamma]$ be an element that is not a commutator in $\Gamma$.
Assume
\begin{enumerate}
\item[(i)] $\rho_\fp(\Gamma)=\SL_2(\F_q)$ with $|\F_q|\ge4$ for every
$\fp\notin E$,
\item[(ii)] the image of $w$ in $\rho_\fp(\Gamma)$ is a commutator
for every $\fp\in E$.
\end{enumerate}
Then the image of $w$ is a commutator in $\rho_\fp(\Gamma)$ for
every maximal ideal $\fp$ of $\cO_{K,S}$. In particular, $w$ witnesses
the failure of the congruence Hasse principle for commutators with
respect to $\rho$.
\end{thm}
\begin{proof}
Let $\fp$ be a maximal ideal of $\cO_{K,S}$. If $\fp\notin E$, every
element of $\rho_\fp(\Gamma)=\SL_2(\F_q)$ is a commutator by
Lemma~\ref{lem:thompson}. If $\fp\in E$, the image of $w$ is a
commutator by (ii).
\end{proof}

\section{The arithmetic setting and the representation}\label{sec:repr}
The field $K=\Q(\sqrt5)$ has ring of integers $\cO_K=\Z[\tau]$, where
$\tau=\tfrac{1+\sqrt5}{2}$ is the golden ratio. This ring is a principal
ideal domain, since the class number of $K$ is $1$. Its fundamental unit is $\tau$, of norm $-1$, so its units
are $\{\pm\tau^m:m\in\Z\}$. Let $S$ be the set of places of $K$ above $2$ and $5$. The prime $5$ ramifies
in $K$, with $(5)=(\sqrt5)^2$, and the prime $2$ is inert. Above each of
$2$ and $5$ there is thus a single prime of $\cO_K$, and adjoining
$\tfrac12$ and $\tfrac15$ inverts both. The ring of $S$-integers is therefore
$$
\cO_{K,S}=\cO_K\bigl[\tfrac12,\tfrac15\bigr]=\Z\bigl[\tau,\tfrac{1}{10}\bigr].
$$

Set $F_0=\langle A_0,B_0\rangle\le\SL_2(\Z)$, where
$$
A_0=\begin{pmatrix}1&2\\0&1\end{pmatrix},\qquad
B_0=\begin{pmatrix}1&0\\2&1\end{pmatrix}.
$$
The group $F_0$ is free of rank $2$ and is called the \emph{Sanov
subgroup}. This is Sanov's theorem~\cite{Sanov1947} and is usually proved by the Ping--Pong Lemma
(see~\cite[Chapter~II.B]{delaHarpe}).

The commutator of the Sanov pair is
\[
J=[A_0,B_0]=\begin{pmatrix}21&-8\\8&-3\end{pmatrix}.
\]
Its characteristic polynomial is $\lambda^2-18\lambda+1$, and its eigenvalues
are $9\pm4\sqrt5=\tau^{\pm6}$. Set
\[
R=\begin{pmatrix}1&1\\ \tau^2 & \tau^{-2}\end{pmatrix}
 =\begin{pmatrix}1&1\\[2pt]
 \dfrac{3+\sqrt5}{2} & \dfrac{3-\sqrt5}{2}\end{pmatrix}
.
\]
Then $\det R=-\sqrt5$ and
\[
\Delta :=R^{-1}JR =\begin{pmatrix}\tau^{-6}&0\\0&\tau^6
\end{pmatrix}=\begin{pmatrix}9-4\sqrt5&0\\0&9+4\sqrt5\end{pmatrix}.
\]

\subsection{The representation $\rho$}

Set
\[
T=\begin{pmatrix}1&0\\0&2\end{pmatrix},
\]
and define
\[
A=(RT)^{-1}A_0(RT),\quad
B=(RT)^{-1}B_0(RT),\quad
C=R^{-1}B_0R,\quad
D=R^{-1}A_0R .
\]
The entries of $R$ lie in $\cO_K\subseteq\cO_{K,S}$, and $\sqrt5$, though
not a unit of $\cO_K$, is a unit of $\cO_{K,S}$, since
$1/\sqrt5=\sqrt5/5$. Hence $\det R=-\sqrt5$ is a unit of $\cO_{K,S}$, and
the entries of $R^{-1}$ lie in $\cO_{K,S}$ as well. The same holds for $T^{-1}$, because $\det T=2$ is also
a unit. Hence all four matrices above lie in $\SL_2(\cO_{K,S})$.

\begin{lem}\label{lem:relation}
The assignment $a_1\mapsto A$, $b_1\mapsto B$,
$a_2\mapsto C$, $b_2\mapsto D$ defines a homomorphism
$\rho\colon\GammaTwo\to\SL_2(\cO_{K,S})$.
\end{lem}
\begin{proof}
The matrices $T$ and $\Delta$ are both diagonal, so they commute. Hence
\[
[A,B]=(RT)^{-1}J(RT)=T^{-1}\Delta T=\Delta ,
\]
and
\[
[C,D]=R^{-1}[B_0,A_0]R=R^{-1}J^{-1}R=\Delta^{-1} .
\]
The product $[A,B][C,D]$ equals $I$, which is the image of the defining
relation of $\GammaTwo$.
\end{proof}

Consequently $\Delta=[A,B]$, and $\rho([a_1,b_1]^k)=\Delta^k$ for every
$k\in\Z$.

\begin{rem}
We choose $T$ diagonal so that it commutes with $\Delta$, and that is
what makes $[A,B]=T^{-1}\Delta T=\Delta$. Only the ratio of the two diagonal
entries of $T$ matters; here the ratio is $2$. The ratio also determines the
coefficient ring, since with ratio $2$ the primes inverted are $2$ and the
$5$ already forced by $R$. With ratio $1$ one would get $A=D$ and $B=C$, and
$\rho$ would factor through the retraction of $\GammaTwo$ onto
$\langle a_1,b_1\rangle$.
\end{rem}

Appendix~\ref{app:faithful} proves that the generic version of the
construction, with the ratio $2$ replaced by an indeterminate $t$, is
faithful.

\begin{cor}\label{cor:separate}
Let $\gamma\in\GammaTwo$. If $\rho_\fp(\gamma)=1$ for every maximal ideal
$\fp$ of $\cO_{K,S}$, then $\rho(\gamma)=I$. Equivalently, the congruence
quotients of $\rho$ separate the points of $\rho(\GammaTwo)$.
\end{cor}
\begin{proof}
Suppose that $\rho_\fp(\gamma)=1$ for every maximal ideal $\fp$ of
$\cO_{K,S}$. Then $\rho(\gamma)\equiv I$ modulo every
maximal ideal, so each entry of $\rho(\gamma)-I$ lies in every maximal
ideal. The ring $\cO_{K,S}$ is a
localisation of the principal ideal domain $\cO_K$, hence is again a
principal ideal domain. In a principal ideal domain a nonzero element has
only finitely many prime divisors, so it lies in only finitely many
maximal ideals. But $\cO_{K,S}$ has infinitely many maximal ideals, one
for each finite place outside $S$. Hence each entry of $\rho(\gamma)-I$ is
zero, and $\rho(\gamma)=I$.
\end{proof}

\subsection{The congruence quotients are large}

The mechanism of \S\ref{sec:mech} requires the congruence quotients to be
large.  We begin with the reduction of the Sanov
subgroup itself.

\begin{lem}\label{lem:sanowmod}
For every odd prime $p$, reduction modulo $p$ carries the Sanov subgroup
$F_0$ onto $\SL_2(\F_p)$.
\end{lem}
\begin{proof}
Write
\[
E_{12}(x)=\begin{pmatrix}1&x\\0&1\end{pmatrix},\qquad
E_{21}(x)=\begin{pmatrix}1&0\\x&1\end{pmatrix}.
\]
Modulo $p$ we have $\bar A_0=E_{12}(2)$ and $\bar B_0=E_{21}(2)$, and
$E_{12}(2)^k=E_{12}(2k)$ as well as $E_{21}(2)^k=E_{21}(2k)$. Since $2$ is
invertible in $\F_p$, the group $\langle\bar A_0,\bar B_0\rangle$ contains
$E_{12}(1)$ and $E_{21}(1)$. It is well known that these two matrices
generate $\SL_2(\F_p)$; see \cite[Chapter~2]{Gorenstein1980}.
\end{proof}

\begin{prop}\label{prop:surj}
Let $\fp\subset\cO_{K,S}$ be a maximal ideal with residue field $\F_q$,
residue characteristic $p$, and reduction map
$\SL_2(\cO_{K,S})\to\SL_2(\F_q)$, $X\mapsto\bar X$. Then $p\ne2,5$ and
\begin{enumerate}
\item[(i)] if $p\equiv\pm1\pmod5$, then $q=p$ and
$\rho_\fp(\GammaTwo)=\SL_2(\F_p)$;
\item[(ii)] if $p\equiv\pm2\pmod5$ and $p\ge7$, then $q=p^2$ and
$\rho_\fp(\GammaTwo)=\SL_2(\F_{p^2})$;
\item[(iii)] if $p=3$, then $q=9$ and
\begin{equation}\label{eq:p3conj}
\rho_\fp(\GammaTwo)=\bar R^{-1}\SL_2(\F_3)\bar R,
\end{equation}
a subgroup of order $24$ in $\SL_2(\F_{9})$.
\end{enumerate}
\end{prop}
\begin{proof}
The primes $2$ and $5$ are inverted in $\cO_{K,S}$, so no maximal ideal of
$\cO_{K,S}$ lies over them. Thus $p\ne2,5$, and $p$ falls into one of three
cases: $p\equiv\pm1\pmod5$; $p\equiv\pm2\pmod5$ with $p\ge7$; or $p=3$,
which also satisfies $3\equiv-2\pmod5$.

The splitting of $p$ in $K=\Q(\sqrt5)$ separates the first two cases. By
quadratic reciprocity, $5$ is a square modulo $p$ if and only if
$p\equiv\pm1\pmod5$. In that case $x^2-x-1$ splits modulo $p$, because its
discriminant is $5$ and $5$ is a square modulo $p$. As $\cO_K=\Z[\tau]$,
the prime $p$ splits in $\cO_K$, and the residue field of $\fp$ is $\F_p$.
Otherwise $x^2-x-1$ is irreducible modulo $p$, so $p$ is inert, the residue
field is $\F_{p^2}$, and $\bar s:=\overline{\sqrt5}$ generates it over
$\F_p$.

(i) Here $\bar s\in\F_p$, so $\bar\tau\in\F_p$; hence the entries of
$\bar R$ and of $\overline{RT}$ lie in $\F_p$, and $\det\bar R=-\bar s\ne0$.
Thus $\bar R$ and $\overline{RT}$ lie in $\GL_2(\F_p)$. By
Lemma~\ref{lem:sanowmod} the reductions $\bar A_0,\bar B_0$ generate
$\SL_2(\F_p)$. Since
$\bar A=(\overline{RT})^{-1}\bar A_0\overline{RT}$ and
$\bar B=(\overline{RT})^{-1}\bar B_0\overline{RT}$, they generate
$(\overline{RT})^{-1}\SL_2(\F_p)\overline{RT}$. This equals $\SL_2(\F_p)$,
because $\overline{RT}$ is an $\F_p$-rational change of basis.

(ii) Here $\bar s\notin\F_p$, so $\F_{p^2}=\F_p\oplus\bar s\F_p$, and
$\SL_2(\F_{p^2})$ contains the subgroup $\SL_2(\F_p)$ of matrices
with entries in $\F_p$. The reductions $\bar A_0,\bar B_0$ have entries in
$\F_p$, so by Lemma~\ref{lem:sanowmod} they generate $\SL_2(\F_p)$. Since
$\bar A=(\overline{RT})^{-1}\bar A_0\overline{RT}$ and
$\bar B=(\overline{RT})^{-1}\bar B_0\overline{RT}$, the images of
$a_1,b_1$ generate
\[
M_1=(\overline{RT})^{-1}\SL_2(\F_p)\overline{RT},
\]
a subgroup of order $p(p^2-1)$ of $\SL_2(\F_{p^2})$.

Let $M_0$ be the normalizer of $\SL_2(\F_p)$ in $\SL_2(\F_{p^2})$, so
that $M_1$ has normalizer $Q=(\overline{RT})^{-1}M_0(\overline{RT})$.

\textbf{Claim 1.}
$
M_0=\{\lambda X:\lambda\in\F_{p^2}^{\times},\ X\in\GL_2(\F_p),\
\det(\lambda X)=1\}.
$

\emph{Proof of Claim 1.} Every matrix of this form lies in $M_0$, since
conjugation by $X\in\GL_2(\F_p)$ preserves $\SL_2(\F_p)$ and $\lambda I$ is
central. Conversely, let
$Y=\begin{pmatrix}a&b\\c&d\end{pmatrix}\in\SL_2(\F_{p^2})$ normalize
$\SL_2(\F_p)$; then $\det Y=1$, so
$Y^{-1}=\begin{pmatrix}d&-b\\-c&a\end{pmatrix}$, and
\[
Y\begin{pmatrix}1&1\\0&1\end{pmatrix}Y^{-1}
=\begin{pmatrix}1-ac&a^2\\-c^2&1+ac\end{pmatrix},\qquad
Y\begin{pmatrix}1&0\\1&1\end{pmatrix}Y^{-1}
=\begin{pmatrix}1+bd&-b^2\\d^2&1-bd\end{pmatrix}
\]
both lie in $\SL_2(\F_p)$. Hence
\[
a^2,\ ac,\ c^2,\ b^2,\ bd,\ d^2\in\F_p .
\]
We show that $a,b,c,d$ lie on a common $\F_p$-line. If $a\ne0$, then
$c/a=(ac)/a^2\in\F_p$, so $c\in a\F_p$; and $a^2\in\F_p$ forces $a\F_p$
to be $\F_p$ or $\bar s\F_p$ (write $a=a_0+a_1\bar s$; then $a_0a_1=0$).
If $a=0$, then $\det Y=-bc=1$, so $c\ne0$, and the same argument applies
to $c$. Thus $a,c\in\lambda_1\F_p$ for some $\lambda_1\in\{1,\bar s\}$.
The same argument, applied to the pair $b,d$, which does not both vanish,
gives $b,d\in\lambda_2\F_p$ for some $\lambda_2\in\{1,\bar s\}$. If
$\lambda_1\ne\lambda_2$, then $ad,bc\in\bar s\F_p$, so
$\det Y=ad-bc\in\bar s\F_p$, contradicting $\det Y=1$. Hence
$\lambda_1=\lambda_2$, and $Y=\lambda_1X$ with $X\in\GL_2(\F_p)$. This
proves the claim.

\textbf{Claim 2.} $\rho_\fp(\GammaTwo)$ contains $M_1$ but not its normalizer
$Q$. In particular, the element
$\bar N=\overline{RT}\,\rho_\fp(a_2)\,(\overline{RT})^{-1}$ does not lie in
$M_0$.

\emph{Proof of Claim 2.} The images of $a_1,b_1$ generate $M_1\subseteq Q$,
so it remains to test the image of $a_2$; by the definition of $\bar N$,
$\rho_\fp(a_2)\in Q$ if and only if $\bar N\in M_0$. Using $\det R=-\sqrt5$
one computes
\[
\bar N
 =\bar R\bar T\bar R^{-1}\bar B_0\bar R\bar T^{-1}\bar R^{-1}
 =\frac1{10}\begin{pmatrix}13-3\sqrt5&-2\\27-9\sqrt5&7+3\sqrt5
\end{pmatrix}.
\]
By Claim~1, $\bar N\in M_0$ would mean $\bar N=\lambda X$ with
$\lambda\in\F_{p^2}^{\times}$ and $X\in\GL_2(\F_p)$, in which case the
quotient of the $(1,1)$-entry of $\bar N$ by the nonzero $(1,2)$-entry
$-1/5$ would lie in $\F_p$ (if $X_{12}=0$, then $\bar N_{12}=0$, which it is
not). But that quotient is $-(13-3\bar s)/2$, whose
coefficient of $\bar s$ is $3/2$, nonzero in $\F_p$; this would force
$\bar s\in\F_p$, contrary to $\bar s\notin\F_p$. Hence $\rho_\fp(a_2)\notin
Q$, and the claim is proved.

\textbf{Claim 3.} Let $p\ge7$. If $\SL_2(\F_p)\le H\le\SL_2(\F_{p^2})$ and
$H\not\subseteq M_0$, then $H=\SL_2(\F_{p^2})$.

\emph{Proof of Claim 3.} Let
$\pi\colon\SL_2(\F_{p^2})\to\PSL_2(\F_{p^2})$ be the quotient by $\{\pm I\}$,
and suppose $H\ne\SL_2(\F_{p^2})$. Since $-I\in\SL_2(\F_p)\subseteq H$, we
have $\pi^{-1}(\pi(H))=H$, so $\pi(H)$ is a proper subgroup of
$\PSL_2(\F_{p^2})$ containing $\PSL_2(\F_p)$. Let $L$ be a maximal subgroup
containing $\pi(H)$.

By Dickson's classification~\cite{Dickson1901} of the subgroups of $\PSL_2(q)$
(see also \cite[Hauptsatz~8.27]{Huppert1967}), the
maximal subgroup $L$ is one of the following candidates up to conjugacy: a dihedral
group of order $p^2-1$ or $p^2+1$; a Borel subgroup, the image of the
upper triangular matrices, of order $p^2(p^2-1)/2$; the subfield subgroup
$\PGL_2(\F_p)$; or one of $A_4$, $S_4$, $A_5$. Indeed, the remaining
types in Dickson's list, elementary abelian $p$-groups and cyclic groups
of order prime to $p$, lie in a Borel subgroup and in a dihedral one,
respectively, and $\PSL_2(\F_p)$ lies in $\PGL_2(\F_p)$. The dihedral
groups have order at most $p^2+1$, and $A_4$, $S_4$, $A_5$ have order at
most $60$. Since $p^2+1<p(p^2-1)/2$ and $60<168\le p(p^2-1)/2$ for
$p\ge7$, both are smaller than $|\PSL_2(\F_p)|=p(p^2-1)/2\le|L|$. A Borel
subgroup is metabelian, hence solvable, whereas $\PSL_2(\F_p)$ is simple
and nonabelian. So $L$ is a conjugate of $\PGL_2(\F_p)$.

By Claim~1 every element of $M_0$ is $\lambda X$ with $X\in\GL_2(\F_p)$. We
embed $\PGL_2(\F_p)$ into $\PSL_2(\F_{p^2})=\SL_2(\F_{p^2})/\{\pm I\}$ by
$[X]\mapsto[X/\sqrt{\det X}]$, where $\sqrt{\det X}\in\F_{p^2}$ exists
because every element of $\F_p^{\times}$ is a square there. Under this
embedding $\lambda X$ and $X$ have the same image: $\det(\lambda X)=1$
gives $\lambda=\pm(\det X)^{-1/2}$, so
$\lambda X=\pm X/\sqrt{\det X}$. Hence $\pi(M_0)$ is the projective image
of $\GL_2(\F_p)$, that is,
$\GL_2(\F_p)/\F_p^{\times}=\PGL_2(\F_p)$. Since $|L|=2|\PSL_2(\F_p)|$
and $L\supseteq\PSL_2(\F_p)$, the group $L$ is the normalizer of
$\PSL_2(\F_p)$ in $\PSL_2(\F_{p^2})$, which is $\pi(M_0)$. Since
$-I\in M_0$, we have
$\pi^{-1}(\pi(M_0))=M_0$, and therefore
$H=\pi^{-1}(\pi(H))\subseteq\pi^{-1}(L)=M_0$. This contradicts the
hypothesis $H\not\subseteq M_0$, so $H=\SL_2(\F_{p^2})$. This proves the
claim.

The group $\overline{RT}\,\rho_\fp(\GammaTwo)\,(\overline{RT})^{-1}$
contains $\SL_2(\F_p)$ and, by Claim~2, an element outside $M_0$, so it is
not contained in $M_0$. By Claim~3 it equals $\SL_2(\F_{p^2})$, and so does
$\rho_\fp(\GammaTwo)$.

(iii) Let $\bar s=\overline{\sqrt5}\in\F_9$. Then $\bar s^2=2$, so
$\F_9=\F_3(\bar s)$ and $\bar s^4=1$. Modulo $3$ we have
$\tau^2\equiv2\bar s$ and $\tau^{-2}\equiv\bar s$, whence
\[
\bar R=\begin{pmatrix}1&1\\2\bar s&\bar s\end{pmatrix},
\qquad \bar T=\begin{pmatrix}1&0\\0&2\end{pmatrix},
\qquad V:=\bar R\bar T\bar R^{-1}=\begin{pmatrix}0&\bar s\\2\bar s&0
\end{pmatrix}.
\]
Here $V^2=I$, because $\bar s\cdot2\bar s=2\bar s^2=4=1$ in $\F_3$. Since
$\bar A=(\bar R\bar T)^{-1}\bar A_0(\bar R\bar T)$, we get
$\bar R\bar A\bar R^{-1}=V^{-1}\bar A_0V=V\bar A_0V$, and similarly for
$\bar B$. Also $\bar R\bar C\bar R^{-1}=\bar B_0$ and
$\bar R\bar D\bar R^{-1}=\bar A_0$. Direct multiplication in
$\SL_2(\F_3)$ gives
\[
V\bar A_0V=\begin{pmatrix}1&0\\1&1\end{pmatrix},\qquad
V\bar B_0V=\begin{pmatrix}1&1\\0&1\end{pmatrix}.
\]
Hence $\bar R\,\rho_\fp(\GammaTwo)\,\bar R^{-1}$ is generated by
$E_{12}(1),E_{21}(1),E_{12}(2),E_{21}(2)$. By Lemma~\ref{lem:sanowmod}
this group is $\SL_2(\F_3)$, which proves \eqref{eq:p3conj}.
\end{proof}

\section{The main theorem}\label{sec:main}

We now assemble the two halves. The group-theoretic half is the retraction
mechanism of \S\ref{sec:mech}; the arithmetic half is
Proposition~\ref{prop:surj}, which says that the congruence quotients are
$\SL_2(\F_q)$ away from the single prime $p=3$.

\begin{thm}\label{thm:infinite}
For every integer $k\ge2$ the element $w_k=[a_1,b_1]^k\in\GammaTwo$ is not a
commutator in $\GammaTwo$, and its image $\rho_\fp(w_k)=\bar\Delta^{\,k}$ is a
commutator in $\rho_\fp(\GammaTwo)$ for \emph{every} maximal ideal
$\fp\subset\cO_{K,S}$.
\end{thm}
\begin{proof}
The first assertion is Corollary~\ref{cor:retr}.

For the second assertion, apply Theorem~\ref{thm:criterion} to $\GammaTwo$,
to the representation $\rho$ of \S\ref{sec:repr}, to $w=w_k$ and to the
set $E$ of maximal ideals of residue characteristic $3$. Condition (i)
holds by Proposition~\ref{prop:surj}(i),(ii), which gives
$\rho_\fp(\GammaTwo)=\SL_2(\F_q)$ with $q\ge7$ for $\fp\notin E$. The prime
$3$ is inert in $K$, so $E$ consists of the single ideal $\fp$ above $3$. For this ideal, $q=9$ and
$\rho_\fp(\GammaTwo)=\bar R^{-1}\SL_2(\F_3)\bar R$ by
Proposition~\ref{prop:surj}(iii). Since $\Delta^2=R^{-1}J^2R$ and
\[
J^2=\begin{pmatrix}377&-144\\144&-55\end{pmatrix}
\equiv\begin{pmatrix}2&0\\0&2\end{pmatrix}=-I\pmod3,
\]
we get $\overline{\Delta^2}=-I$. Also
\[
\bar R\,\overline{\Delta}\,\bar R^{-1}=\bar J
=\begin{pmatrix}0&1\\-1&0\end{pmatrix}=:u,
\]
which has order $4$. Put
\[
v=\begin{pmatrix}1&1\\1&-1\end{pmatrix},\quad
u'=\begin{pmatrix}0&1\\-1&-1\end{pmatrix},\quad
v'=\begin{pmatrix}0&1\\-1&1\end{pmatrix}.
\]
Their determinants are $\det v=\det v'=-2\equiv1$ and $\det u=\det u'=1$ in
$\F_3$, so all four matrices lie in $\SL_2(\F_3)$. Direct multiplication
gives $[u,v]=-I$, $[u',v']=u$ and $[v',u']=-u$. The matrix $\bar R$ is invertible, since
$\det\bar R=-\bar s$ and $\bar s^2=2\ne0$ in $\F_9$. By \eqref{eq:p3conj},
conjugation by $\bar R^{-1}$ maps $\SL_2(\F_3)$ onto
$\rho_\fp(\GammaTwo)$, and it preserves commutators. The powers of $u$ are
$I$, $u$, $-I$, $-u$, and each is a commutator of $\SL_2(\F_3)$:
$I=[I,I]$, $-I=[u,v]$, $u=[u',v']$ and $-u=[v',u']$. Since
$\bar\Delta=\bar R^{-1}u\bar R$, we have $\bar\Delta^{\,k}=\bar
R^{-1}u^k\bar R$, and $u^k$ is a commutator of $\SL_2(\F_3)$ for every
$k\ge2$. Hence $\bar\Delta^{\,k}$ is a commutator of
$\rho_\fp(\GammaTwo)$. This is condition (ii), so
Theorem~\ref{thm:criterion} gives the second assertion.
\end{proof}

The above infinitely many elements of $\GammaTwo$ witness the failure of the
congruence Hasse principle for commutators.
The failure descends to every genus. Let $g\ge2$, and let
\[
\psi\colon\GammaTwo\longrightarrow \Z_{g-1}
\]
be the homomorphism given by $\psi(a_1)=1$ and
$\psi(b_1)=\psi(a_2)=\psi(b_2)=0$. The subgroup $\ker\psi$ is a normal
subgroup of index $g-1$ in $\GammaTwo$, so it corresponds to a regular
covering of $\Sigma_2$ of degree $g-1$. Since $\Sigma_2$ has Euler
characteristic $-2$, the covering surface has Euler
characteristic $-2(g-1)=2-2g$ and hence genus $g$. Its fundamental group
is $\ker\psi$. Choose an isomorphism $i\colon\Gamma_g\to\ker\psi$, and
restrict $\rho$ to get a representation
$\rho\circ i\colon\Gamma_g\to\SL_2(\cO_{K,S})$. This gives congruence
quotients of $\Gamma_g$. We will show that they coincide with the congruence
quotients of $\GammaTwo$.

\begin{prop}\label{prop:kernel}
For every maximal ideal $\fp\subset\cO_{K,S}$ one has
\begin{equation}\label{eq:kernelp}
\rho_\fp(\ker\psi)=\rho_\fp(\GammaTwo).
\end{equation}
\end{prop}
\begin{proof}
Note that $b_1,a_2,b_2\in\ker\psi$, so $\bar B,\bar C,\bar D$ lie in
$\rho_\fp(\ker\psi)$. Let $p$ be the residue characteristic. As in
Proposition~\ref{prop:surj}, $p\equiv\pm1\pmod5$, or $p\equiv\pm2\pmod5$
with $p\ge7$, or $p=3$.

Suppose $p\equiv\pm1\pmod5$. Then $\bar C,\bar D$ generate
$\bar R^{-1}\SL_2(\F_p)\bar R=\SL_2(\F_p)$, where the equality holds
because $\bar R\in\GL_2(\F_p)$. Since
$\rho_\fp(\ker\psi)\subseteq\rho_\fp(\GammaTwo)=\SL_2(\F_p)$ by
Proposition~\ref{prop:surj}(i), equality holds in \eqref{eq:kernelp}.

Suppose $p\equiv\pm2\pmod5$ and $p\ge7$. Then $\bar C,\bar D$ generate
$M_1'=\bar R^{-1}\SL_2(\F_p)\bar R$, whose normalizer is
$Q'=\bar R^{-1}M_0\bar R$. Claim~3 in the proof of
Proposition~\ref{prop:surj}(ii), conjugated by $\bar R^{-1}$, shows that a
subgroup of $\SL_2(\F_{p^2})$ that contains $M_1'$ lies in $Q'$, or equals
$\SL_2(\F_{p^2})$. We claim that $\bar B\notin Q'$. Put
$V=\bar R\bar T\bar R^{-1}$. Then
$\bar R\bar B\bar R^{-1}=V^{-1}\bar B_0V$, and a direct computation gives
\[
V^{-1}\bar B_0V=\frac1{10}\begin{pmatrix}13+3\sqrt5&-2\\
27+9\sqrt5&7-3\sqrt5\end{pmatrix}.
\]
If $\bar B\in Q'$, then $V^{-1}\bar B_0V\in M_0$, so by the description of
$M_0$ it equals $\lambda X$ with $X\in\GL_2(\F_p)$; then the quotient of the
$(1,1)$-entry by the nonzero $(1,2)$-entry $-2/10$ would lie in $\F_p$ (again
$X_{12}\ne0$). But
that quotient is $-(13+3\bar s)/2$, whose coefficient of $\bar s$ is $-3/2$,
nonzero in $\F_p$; this would force $\bar s\in\F_p$, a contradiction. Hence
$\bar B\notin Q'$, and $\rho_\fp(\ker\psi)=\SL_2(\F_{p^2})$.

Suppose $p=3$, and put again $V=\bar R\bar T\bar R^{-1}$. Then $V^2=I$ as
in Proposition~\ref{prop:surj}(iii). The computation there, applied to
$\bar B,\bar C,\bar D$, shows that $\bar R\rho_\fp(\ker\psi)\bar R^{-1}$
contains $\langle V\bar B_0V,\bar B_0,\bar A_0\rangle=\SL_2(\F_3)$; the
last equality is Lemma~\ref{lem:sanowmod}. Since
$\rho_\fp(\ker\psi)\subseteq\rho_\fp(\GammaTwo)
=\bar R^{-1}\SL_2(\F_3)\bar R$ by Proposition~\ref{prop:surj}(iii),
equality holds in \eqref{eq:kernelp}.
\end{proof}

\begin{cor}\label{cor:allgenera}
For every $g\ge2$ the congruence Hasse principle for commutators fails for
$\Gamma_g$, with respect to the representation $\rho\circ i$ above.
\end{cor}
\begin{proof}
Let $k\ge2$. The element $[a_1,b_1]^k$ has exponent sum $0$ in $a_1$,
so it lies in $\ker\psi$. Let $h\in\Gamma_g$ be the element with
$i(h)=[a_1,b_1]^k$. 

Since $a_2,b_2\in\ker\psi$ and
$[a_1,b_1]=[a_2,b_2]^{-1}$, we have $[a_1,b_1]\in[\ker\psi,\ker\psi]
=i([\Gamma_g,\Gamma_g])$, and hence $h\in[\Gamma_g,\Gamma_g]$. If
$h=[x,y]$ in $\Gamma_g$, then $[a_1,b_1]^k=i(h)=[i(x),i(y)]$ in
$\GammaTwo$, which contradicts Corollary~\ref{cor:retr}. Thus $h$ is
not a commutator in $\Gamma_g$.

On the other hand $(\rho\circ i)(h)=\rho([a_1,b_1]^k)=\Delta^k$. For
every maximal ideal $\fp$ the congruence quotient of $\Gamma_g$ at
$\fp$ is $(\rho\circ i)_\fp(\Gamma_g)=\rho_\fp(\ker\psi)
=\rho_\fp(\GammaTwo)$ by Proposition~\ref{prop:kernel}, and the image
of $h$ in this quotient is $\bar\Delta^{\,k}$, a commutator by
Theorem~\ref{thm:infinite}. Since $k\ge2$ was arbitrary and distinct
$k$ give distinct elements $h$, the failure is witnessed by infinitely
many elements of $\Gamma_g$.
\end{proof}


\appendix
\section{Faithfulness of the generic representation}\label{app:faithful}

Let $t$ be an indeterminate, put
\[
T_t=\begin{pmatrix}1&0\\0&t\end{pmatrix},
\qquad
\Lambda=\cO_{K,S}[t,t^{-1}],
\]
and recall from Section~\ref{sec:repr} the Sanov subgroup
$F_0=\langle A_0,B_0\rangle\le\SL_2(\Z)$, its commutator
$J=[A_0,B_0]$, and $\Delta=R^{-1}JR$. Set
\[
A_t=(RT_t)^{-1}A_0(RT_t),\qquad
B_t=(RT_t)^{-1}B_0(RT_t),
\]
and set $C=R^{-1}B_0R$, $D=R^{-1}A_0R$ as in Section~\ref{sec:repr}.
All four matrices lie in
$\SL_2(\Lambda)$. As in the proof of Lemma~\ref{lem:relation}, the
matrices $T_t$ and $\Delta$ are diagonal and commute, so
$[A_t,B_t]=\Delta$ and $[C,D]=\Delta^{-1}$, and the assignment
$a_1\mapsto A_t$, $b_1\mapsto B_t$, $a_2\mapsto C$, $b_2\mapsto D$
defines a homomorphism
\[
\rho_t\colon\GammaTwo\to\SL_2(\Lambda).
\]
The representation $\rho$ of Section~\ref{sec:repr} is $\rho_2$. In this
appendix we prove that $\rho_t$ is faithful.

\begin{lem}\label{lem:primitive}
The commutator $[a,b]$ is not a proper power in the free group
$F_2=\langle a,b\rangle$.
\end{lem}
\begin{proof}
The word $[a,b]=aba^{-1}b^{-1}$ is cyclically reduced of length $4$. We
use the standard fact that if a cyclically reduced word $w$ equals
$v^k$, then $v$ is itself cyclically reduced, so the reduced word for
$w$ is the concatenation of $k$ copies of the reduced word for $v$.
Suppose $[a,b]=v^k$ with $k\ge2$; then $k$ divides $4$, so $k=2$ or $4$. If
$k=4$, then $v$ has length $1$, and $v^4$ is $a^{\pm4}$ or
$b^{\pm4}$, not $[a,b]$. If $k=2$, then $v$ has length $2$, and the
first two letters of $[a,b]$ are those of $v$, so $v=ab$; but
$v^2=abab$ differs from $[a,b]=aba^{-1}b^{-1}$, a
contradiction.
\end{proof}

\begin{lem}\label{lem:nonzero}
For every $w\in F_0$ with $w\notin\langle J\rangle$, both off-diagonal
entries of $R^{-1}wR$ are nonzero.
\end{lem}
\begin{proof}
Suppose $(R^{-1}wR)_{21}=0$, and let $H$ be the group of all $v\in F_0$
for which $R^{-1}vR$ is upper triangular. The map $v\mapsto R^{-1}vR$
embeds $H$ into the upper triangular subgroup of $\SL_2(K)$, which is
solvable, so $H$ is solvable; as a subgroup of the free group $F_0$, it
is cyclic. It contains $J$, because $\Delta=R^{-1}JR$ is diagonal.
Since $F_0$ is free on $A_0$ and $B_0$, Lemma~\ref{lem:primitive} shows
that $J=[A_0,B_0]$ is not a proper power in $F_0$, so the cyclic group
$H$ equals $\langle J\rangle$; thus $w\in\langle J\rangle$, a
contradiction. The same argument with lower triangular matrices gives
the assertion for the $(1,2)$-entry.
\end{proof}

The group $\GammaTwo$ decomposes as the free product of its subgroups
$\langle a_1,b_1\rangle$ and $\langle a_2,b_2\rangle$ with amalgamated
subgroup $\langle c\rangle$, where $c=[a_1,b_1]=[a_2,b_2]^{-1}$. The images of the two free factors are the conjugates
$T_t^{-1}(R^{-1}F_0R)T_t$ and
$R^{-1}F_0R$ of the Sanov subgroup, free on the pairs $A_t,B_t$ and
$C,D$ respectively; hence the restrictions of $\rho_t$ to the two
factors are faithful.

By the
normal form theorem for amalgamated products
\cite[Chapter~IV]{LyndonSchupp}, every $1\neq g\in
\GammaTwo$ is conjugate into $\langle a_1,b_1\rangle$ or
$\langle a_2,b_2\rangle$, or to a product $g_1g_2\cdots g_{2l}$ with
$l\ge1$, $g_{2i-1}\in\langle a_1,b_1\rangle\setminus\langle c\rangle$
and $g_{2i}\in\langle a_2,b_2\rangle\setminus\langle c\rangle$. The entries of every matrix
$\rho_t(g)$ are Laurent polynomials in $t$.

\begin{lem}\label{lem:degtr}
Let $g=g_1\cdots g_{2l}$, $l\ge1$, be in normal form. Then
$\deg\tr\rho_t(g)=l$.
\end{lem}
\begin{proof}
Since $\rho_t(g_{2i-1})\in T_t^{-1}(R^{-1}F_0R)T_t,~
\rho_t(g_{2i})\in R^{-1}F_0R$, we can write
$\rho_t(g_{2i-1})=T_t^{-1}M_{2i-1}T_t$ and $\rho_t(g_{2i})=M_{2i}$ with
$M_{2i-1}=R^{-1}w_{2i-1}R$ and $M_{2i}=R^{-1}w_{2i}R$ for some
$w_{2i-1},w_{2i}\in F_0$; their entries do not involve $t$. Moreover
$w_{2i-1},w_{2i}\notin\langle J\rangle$: indeed $g_{2i-1},g_{2i}\notin
\langle c\rangle$ by the normal form, and $w_{2i-1}=J^k$ would give
$\rho_t(g_{2i-1})=\Delta^k=\rho_t(c^k)$, hence $g_{2i-1}=c^k$ by the
injectivity of $\rho_t$ on the factors, a contradiction; the case
$w_{2i}=J^k$ is similar. 

Put $P_0=I$ and $P_k=\rho_t(g_1\cdots g_{2k})$ for $1\le k\le l$, and
write $P_{k+1}=P_kQ_k$ for $0\le k\le l-1$ with
$Q_k=\rho_t(g_{2k+1}g_{2k+2})
=T_t^{-1}M_{2k+1}T_t\,M_{2k+2}$. The entries of $Q_k$ are
\[
\begin{aligned}
(Q_k)_{11}&=(M_{2k+1})_{11}(M_{2k+2})_{11}
+t\,(M_{2k+1})_{12}(M_{2k+2})_{21},\\
(Q_k)_{12}&=(M_{2k+1})_{11}(M_{2k+2})_{12}
+t\,(M_{2k+1})_{12}(M_{2k+2})_{22},\\
(Q_k)_{21}&=t^{-1}(M_{2k+1})_{21}(M_{2k+2})_{11}
+(M_{2k+1})_{22}(M_{2k+2})_{21},\\
(Q_k)_{22}&=t^{-1}(M_{2k+1})_{21}(M_{2k+2})_{12}
+(M_{2k+1})_{22}(M_{2k+2})_{22}.
\end{aligned}
\]
Thus $\deg(Q_k)_{12}\le1$ and $\deg(Q_k)_{21},\deg(Q_k)_{22}\le0$,
while $\deg(Q_k)_{11}=1$ since
$(M_{2k+1})_{12}(M_{2k+2})_{21}$ is nonzero by
Lemma~\ref{lem:nonzero}. 

We prove by induction on $k\ge1$ that
\[
\deg(P_k)_{11}=k,\qquad \deg(P_k)_{12}\le k,\qquad
\deg(P_k)_{21}\le k-1,\qquad \deg(P_k)_{22}\le k-1.
\]
For $k=1$, where $P_1=Q_0$, these are exactly the bounds on $Q_0$ just
proved. Assume them for some $k\ge1$; the
entries of $P_{k+1}=P_kQ_k$ are
\[
\begin{aligned}
(P_{k+1})_{11}&=(P_k)_{11}(Q_k)_{11}+(P_k)_{12}(Q_k)_{21},\\
(P_{k+1})_{12}&=(P_k)_{11}(Q_k)_{12}+(P_k)_{12}(Q_k)_{22},\\
(P_{k+1})_{21}&=(P_k)_{21}(Q_k)_{11}+(P_k)_{22}(Q_k)_{21},\\
(P_{k+1})_{22}&=(P_k)_{21}(Q_k)_{12}+(P_k)_{22}(Q_k)_{22}.
\end{aligned}
\]
The entries of $P_k$ have degree at most $k$ in the first row and at
most $k-1$ in the second row, and those of $Q_k$ at most $1$ in the first
row and at most $0$ in the second; hence $\deg(P_{k+1})_{12}\le k+1$
and $\deg(P_{k+1})_{21},\deg(P_{k+1})_{22}\le k$. In the $(1,1)$-entry
the summand $(P_k)_{12}(Q_k)_{21}$ has degree at most $k$, while
$(P_k)_{11}(Q_k)_{11}$ has degree exactly $k+1$; hence
$\deg(P_{k+1})_{11}=k+1$, which proves the claim for $k+1$. 

Taking
$k=l$, the $(1,1)$-entry of $\rho_t(g)=P_l$ has degree $l$, while its
$(2,2)$-entry has degree at most $l-1$. Hence $\deg\tr\rho_t(g)=l$.
\end{proof}

\begin{thm}\label{thm:faithful}
The homomorphism $\rho_t\colon\GammaTwo\to\SL_2(\Lambda)$ is faithful.
\end{thm}
\begin{proof}
Let $g\ne1$. Conjugation by an element of $\GammaTwo$ does not affect
whether $\rho_t(g)$ is the identity, so by the normal form theorem there
are two cases.

If $g$ is conjugate to an element of one of the free factors, then
$\rho_t(g)$ is conjugate to the image of a nontrivial element of that
factor. The restrictions of $\rho_t$ to the two factors are faithful,
so $\rho_t(g)\ne I$.

Otherwise $g$ is conjugate to a product $g_1\cdots g_{2l}$, $l\ge1$, in
normal form. By Lemma~\ref{lem:degtr} we have $\deg\tr\rho_t(g)=l\ge1$. Hence
$\tr\rho_t(g)$ is not the constant $2$, and $\rho_t(g)\ne I$. 

Thus $\rho_t$ is injective.
\end{proof}

\end{document}